\documentclass[twocolumn]{autart}
\usepackage{graphicx,amsmath,amsfonts,amssymb,mathrsfs,bm,tabularx,array,enumerate,algorithm,algorithmic, pifont,color,tikz,harvard,fancybox,multirow}

\newtheorem{lemma}{Lemma}
\newtheorem{theorem}{Theorem}

\newtheorem{assumption}{Assumption}

\newtheorem{remark}{Remark}

\DeclareMathOperator{\diag}{diag}

\DeclareMathOperator{\argmin}{argmin}

\DeclareMathOperator{\VI}{VI}
\DeclareMathOperator{\SOL}{SOL}

\DeclareMathOperator{\LCP}{LCP}

\newenvironment{proof}{{\bf Proof.}}{\hfill $\square$}
\begin{document}

\begin{frontmatter}

\title{Distributed sub-optimal resource allocation via a projected form of singular perturbation\thanksref{footnoteinfo}} 
\thanks[footnoteinfo]{The material in this paper was not presented at any conference.}

\author[USTB]{Shu Liang} \ead{sliang@ustb.edu.cn},
\author[BIT]{Xianlin Zeng}\ead{xianlin.zeng@bit.edu.cn},
\author[CAS]{Guanpu Chen}\ead{chengp@amss.ac.cn},
\author[CAS]{Yiguang Hong}\ead{yghong@iss.ac.cn}

\address[USTB]{Key Laboratory of Knowledge Automation for Industrial Processes of Ministry of Education, School of Automation and Electrical Engineering, University of Science and Technology Beijing, Beijing 100083, China}

\address[BIT]{Key Laboratory of Intelligent Control and Decision of Complex Systems, School of Automation, Beijing Institute of Technology, 100081, Beijing, China}
\address[CAS]{Key Laboratory of Systems and Control, Academy of
Mathematics and Systems Science, Chinese Academy of Sciences,
Beijing, 100190, China}

\begin{keyword}
Distributed optimization, resource allocation, sub-optimality, weight-balanced graph, singular perturbation
\end{keyword}

\begin{abstract}
Distributed optimization for resource allocation problems is investigated and a sub-optimal continuous-time algorithm is proposed. Our algorithm has lower order dynamics than others to reduce burdens of computation and communication, and is applicable to weight-balanced graphs. Moreover, it can deal with both local set constraints and coupled inequality constraints, and remove the requirement of twice differentiability of the cost function in comparison with the existing sub-optimal algorithm. However, this algorithm is not easy to be analyzed since it involves singular perturbation type dynamics with projected non-differentiable right-hand side. We overcome the encountered difficulties and obtain results including the existence of an equilibrium, the sub-optimality, and the convergence of the algorithm.
\end{abstract}
\end{frontmatter}
\section{Introduction}
{\color{white}
\cite{Xiao2006Optimal,Lakshmanan2008Decentralized,Nedic2018Resource,Yuan2018Adaptive,Zhu2019Resource,Xu2019Regularization}
\cite{Halabian2019D5G,Bandi2018Robust,Yang2017Distributed,Cherukuri2015Distributed}
\cite{Cherukuri2016Initialization,Yi2016Initialization,Yun2019Initialization,Liang2018Distributed}
\cite{Kia2017Distributed,Deng2018Distributed,Johansson2008Subgradient,Liang2018Singular,Kokotovic1999Singular}
}\vspace{-110pt}\\
Recently, distributed multi-agent resource allocation optimization has received much attention from various fields such as control and optimization \cite{Xiao2006Optimal,Lakshmanan2008Decentralized,Nedic2018Resource,Yuan2018Adaptive,Zhu2019Resource,Xu2019Regularization}, communication \cite{Halabian2019D5G}, management \cite{Bandi2018Robust}, and power system \cite{Yang2017Distributed}. Many continuous-time algorithms have been developed to solve these problems. For a brief review, a Laplacian-gradient dynamics has been presented in \cite{Cherukuri2015Distributed}, while initialization-free algorithms have been introduced in \cite{Cherukuri2016Initialization,Yi2016Initialization,Yun2019Initialization}. In particular, algorithms given in \cite{Cherukuri2016Initialization,Yi2016Initialization} are based on primal-dual gradient flows, while the algorithm introduced in \cite{Yun2019Initialization} is based on dual gradient. In addition, a distributed algorithm dealing with coupled inequality constraints has been proposed in \cite{Liang2018Distributed} via a modified Lagrangian function.

Network topology is an essential part in distributed algorithm design and analysis. Many distributed algorithms for resource allocation problems rely on undirected graphs, such as \cite{Xiao2006Optimal,Lakshmanan2008Decentralized,Yi2016Initialization,Liang2018Distributed,Yun2019Initialization}. It is well-known that balanced digraphs are less restrictive and more general than undirected graphs. A few works such as \cite{Cherukuri2016Initialization,Kia2017Distributed,Deng2018Distributed} have considered weight-balanced graphs for resource allocation problems, but their methods need additional computation for the spectral information of the Laplacians.

Sub-optimal solution is sometimes preferable because it may simplify algorithm design and reduce the cost of computation. For example, \cite{Johansson2008Subgradient} has developed a simple distributed algorithm to solve an optimal consensus problem and obtained an sub-optimal solution. How can the sub-optimal concept further serve distributed optimization? It is known that distributed algorithms get involved with networks for information sharing, where local ``uncoordinated'' flows must be compensated for the desired optimality.
It becomes much difficult for directed graphs, because an unidirectional flow can only be compensated by others in the network. With these observations, \cite{Liang2018Singular} has presented a simple distributed algorithm for a special resource allocation problem via singular perturbation, which reduces computation and communication burdens and obtains a sub-optimal solution.

In this paper, we propose a projected singular perturbation dynamics for resource allocation problems with local set constraints and coupled inequality constraints. Although the idea originates from \cite{Liang2018Singular}, the previous analysis is not applicable to our new algorithm. One reason is that singular perturbation analysis provides first few terms in the Taylor expansion of the trajectory, which requires at least continuous differentiability on the right-hand side of the differential equation \cite{Kokotovic1999Singular}. However, due to the presence of projection in both fast and slow dynamics, the differentiability does not hold. In fact, it is even difficult to ascertain the existence of an equilibrium and its stability and optimality. To overcome these, we employ theories from linear complementarity problems and variational inequalities, and treat the primal and dual parts as two interacted static systems: the former is a perturbed variational inequality problem and the latter is a perturbed complementarity problem. The main contributions of this work are summarized as follows.
\begin{enumerate}[1)]
\item A distributed singular perturbation type dynamics is developed to solve resource allocation problems with local set constraints and coupled inequality constraints over weight-balanced graphs, whereas \cite{Liang2018Singular} deals with a special problem with coupled equality constraints only.
\item New analysis methods for the equilibrium, sub-optimality and convergence are provided, which deal with a challenging problem involving singular perturbation dynamics with non-differentiable right-hand side. In addition, the assumption on the twice continuous differentiability of the cost function is relaxed.
\item Our algorithm uses local primal and dual variables without any auxiliary variable. Therefore, it has lower order dynamics than those in \cite{Cherukuri2016Initialization,Kia2017Distributed,Deng2018Distributed}, and reduces the computation and communication burden.
\end{enumerate}


\section{Preliminaries}
\label{sec:preliminaries}
In this section, we give the basic notations and introduce preliminary knowledge about convex analysis, variational inequalities, and graph theory.

$\mathbb{R}^n$ is the $n$-dimensional real vector space and $\mathbb{R}^n_+$ is the nonnegative orthant. $I_n$ is the unit matrix in $\mathbb{R}^{n\times n}$.
$\|\cdot\|$ is the Euclidean norm and $\mathbb{B}$ is the unit ball in a Euclidean space.
$\otimes$ is the operator of Kroneckor's product.
$col(x_1,...,x_n)$ is the column vector stacked with column vectors $x_1,...,x_n$.
For a vector $a\in \mathbb{R}^n$, $a \leq\bm{0}$ (or $a < \bm{0}$) means that each component of $a$ is less than or equal to zero (or smaller than zero). For vectors $a,b \in \mathbb{R}^n$, $a\perp b$ means that $a^Tb = 0$.

For a closed convex set $C$, the {\em projection} map $P_{C}:\mathbb{R}^n \to C$ is defined as $P_{C}(x) \triangleq \mathop{\argmin}_{y\in C} \|x-y\|$. Two basic properties with respect to the projection operator hold:
\begin{align}
\label{eq:pro_projection}
&(x-P_{C}(x))^{T}(P_{C}(x)-y) \geq 0, \quad \forall \, y\in C\text{,}\\
\label{eq:pro_Lip}
&\|P_{C}(x)-P_{C}(y)\|\leq \|x-y\|, \quad \forall \, x,y\in \mathbb{R}^n\text{.}
\end{align}
For $x\in C$, the {\em tangent cone} to $C$ at $x$ is $\mathcal{T}_C(x) \triangleq \{\lim_{k\to\infty}\frac{x_k-x}{t_k}\,|\,x_k\in C, t_k>0,\text{ and }x_k \to x, t_k\to 0\}$, and the {\em normal cone} to $C$ at $x$ is $\mathcal{N}_C(x) \triangleq \{v\in \mathbb{R}^n\,|\,v^T(y -x) \leq 0, \text{ for all } y\in C\}$.

A differentiable function $f: C\to \mathbb{R}$ is said to be $\mu$-strongly convex for some constant $\mu \geq 0$ if $(x-y)^T(\nabla f(x) - \nabla f(y)) \geq \mu \|x - y\|^2,\, \forall\,x,y\in C$. In other words, $\nabla f$ is $\mu$-strongly monotone.

Given a subset $\Omega \subseteq \mathbb{R}^n$ and a map $F: \Omega \to \mathbb{R}^n$, the problem of {\em variational inequality}, denoted by $\VI(\Omega,F)$,
is to find a vector $x\in \Omega$ such that
\begin{equation*}
(y-x)^TF(x) \geq 0,\quad \forall\, y\in \Omega,
\end{equation*}
and the set of solutions is denoted by $\SOL(\Omega,F)$.
When $\Omega$ is closed and convex, the solution of $\VI(\Omega,F)$ can be equivalently reformulated via projection or the normal cone \cite{Facchinei2003Finite}:
\begin{equation}\label{eq:SOL2Pro}
\begin{aligned}
x\in \SOL(\Omega,F) & \Longleftrightarrow  \bm{0} = P_\Omega(x-F(x)) - x\\
& \Longleftrightarrow \bm{0} \in F(x) + \mathcal{N}_\Omega(x)
\end{aligned}
\end{equation}
In particular, if $\Omega = \mathbb{R}^n_+$ and $F(x) = q + Mx$ for some vector $q\in \mathbb{R}^n$ and matrix $M\in\mathbb{R}^{n\times n}$, then the variational inequality becomes so-called {\em linear complementarity problem}, denoted by $\LCP(q, M)$, with its solution set denoted by $\SOL(q,M)$.

Consider a network topology described by a weighted graph $\mathcal{G} = \{\mathcal{V}, \mathcal{E}, \mathcal{A}\}$, where $\mathcal{V}= \{ 1,2, \ldots N\}$ is the node set, $\mathcal{E} \subseteq \mathcal{V} \times \mathcal{V}$ is the (oriented) edge set, and $\mathcal{A} = [a_{ij}] \in \mathbb{R}^{N\times N}$ is a nonnegative weight matrix. An edge $(j,i) \in \mathcal{E}$ means that node $j$ can send its information to node $i$. In this case, node $j$ is said to be an in-neighbor of node $i$. The set of all in-neighbors of node $i$ is denoted by $\mathcal{N}_i$. Also, $a_{ij} > 0 $ if $j \in \mathcal{N}_i$, while $a_{ij} = 0$ otherwise. A path is a sequence of vertices connected by edges.
$\mathcal{G}$ is said to be {\em strongly connected} if there is a path between any pair of vertices.
$\mathcal{G}$ is said to be {\em weight-balanced} if for every $i\in\mathcal{V}$, $\sum_{j = 1}^N a_{ij} = \sum_{j = 1}^N a_{ji} = d_i$.
The Laplacian matrix of the weight-balanced $\mathcal{G}$ is $L= \mathcal{D} - \mathcal{A}$, where $\mathcal{D} = \diag\{d_1, ...,  d_N\} \in \mathbb{R}^{N \times N}$. If $\mathcal{G}$ is strongly connected and weight-balanced, then $L + L^T$ is positive semidefinite and $0$ is its simple eigenvalue.
\section{Formulation and algorithm}
\label{sec:problem}
In this section, we formulate the distributed resource allocation problem and present our distributed algorithm.

\subsection{Problem formulation}
Consider a multi-agent network with graph $\mathcal{G} = \{\mathcal{V}, \mathcal{E}, \mathcal{A}\}$. For each $i\in \mathcal{V}$, the $i$th agent has a decision variable $x_i$ in a local feasible set $\Omega_i\subset \mathbb{R}^{n_i}$. Also, it has a cost function $f_i: \Omega\to \mathbb{R}$ and a resource map $g_i: \Omega_i \to \mathbb{R}^p$. Define
\begin{equation*}
\bm{x} \triangleq col(x_1, x_2, ..., x_N),\quad \bm{\Omega} \triangleq \Omega_1 \times \Omega_2 \times \cdots \times \Omega_N,\\
\end{equation*}
and the total cost function and resource map
\begin{align*}
f(\bm{x}) & \triangleq f_1(x_1) + f_2(x_2) + \cdots + f_N(x_N),\\
g(\bm{x}) & \triangleq g_1(x_1) + g_2(x_2) + \cdots + g_N(x_N).
\end{align*}
Then the resource allocation problem with coupled inequality constraints can be formulated as
\begin{equation}\label{eq:optimizationProblem}
\min_{\bm{x} \in \bm{\Omega}} f(\bm{x}), \,\, \text{ s.t. } \,\, g(\bm{x})\leq \bm{0}.
\end{equation}
Our goal is to design a distributed algorithm for problem \eqref{eq:optimizationProblem} and find some sub-optimal solution. Of course, the
design of sub-optimal algorithms should be simpler than those for optimal solutions. We introduce Assumption \ref{assum:1} for the considered distributed optimization problem.
\begin{assumption}\label{assum:1}
~{\em
\begin{itemize}
\item (Objective function) For each $i\in\mathcal{V}$, $f_i$ is $\mu_f$-strongly convex over $\Omega_i$ for some constant $\mu_f>0$, and $\nabla f_i$ is $\kappa_f$-Lipschitz continuous over $\Omega_i$ for some $\kappa_f >0$.
\item (Constraint set and function) For each $i\in\mathcal{V}$, $\Omega_i$ is closed and convex, and $g_i$ is convex and $\kappa_g$-Lipschitz continuous over $\Omega_i$ for some constant $\kappa_g>0$. Also, $\nabla g_i $ is locally Lipschitz continuous over $\Omega_i$.
\item (Slator's constraint qualification) There exists a vector $\tilde{\bm{x}}$ that belongs to the relative interior of $\bm{\Omega}$ and satisfies $g(\tilde{\bm{x}}) < \bm{0}$.
\item (Network topology) Graph $\mathcal{G}$ is strongly connected and weight-balanced.
\end{itemize}}
\end{assumption}

The convexity of the cost and constraint functions ensures that \eqref{eq:optimizationProblem} is a convex optimization problem.
The smoothness enables the use of gradient and the constraint qualification ensures first-order necessary conditions.
These assumptions are basic and widely used for constrained convex optimizations \cite{Luenberger2016Linear}. The strong connectivity and weight-balance of the network are the same as those in  \cite{Cherukuri2016Initialization,Kia2017Distributed,Liang2018Singular,Deng2018Distributed}.


\subsection{Distributed algorithm}
Our algorithm for problem \eqref{eq:optimizationProblem} is given as follows.
\begin{algorithm}
\caption{(for each $i\in\mathcal{V}$)}
\label{alg:2}
\textbf{Initialization}:
\begin{equation*}
x_i(0) \in \Omega_i,\quad \lambda_i(0) \in \mathbb{R}_+^p.
\end{equation*}
\textbf{Update flows}:
\begin{equation}\label{eq:Algorithm2}
\left\{%
\begin{aligned}
\dot{x}_i & = P_{\Omega_i}(x_i - \nabla f_i(x_i) - \nabla g_i(x_i)\lambda_i) - x_i\\
\varepsilon \dot{\lambda}_i & = \max\big\{-\varepsilon \lambda_i, \,\, \varepsilon g_i(x_i) - \sum_{j \in \mathcal{N}_i}a_{ij}(\lambda_i - \lambda_j)\big\}
\end{aligned}\right.
\end{equation}
where $\varepsilon > 0$ is a small tunable parameter.
\end{algorithm}

Algorithm \ref{alg:2} is distributed since the update flows of the $i$th agent need only $x_i$, $\lambda_i$, $\nabla f_i(x_i)$, $\nabla g_i(x_i)$ and the neighbors' $\lambda_j$.
The compact form of \eqref{eq:Algorithm2} can be written as
\begin{equation}\label{eq:Algorithmcompact2}
\left\{%
\begin{aligned}
\dot{\bm{x}} & =  P_{\bm{\Omega}}(\bm{x} - \nabla f(\bm{x}) - \bm{v}(\bm{x},\bm{\lambda})) - \bm{x}\\
\varepsilon \dot{\bm{\lambda}} & = P_{\mathbb{R}^{pN}_+}(\varepsilon \bm{\lambda} + \varepsilon \bm{u}(\bm{x}) - \bm{L} \bm{\lambda}) - \varepsilon \bm{\lambda}
\end{aligned}\right.
\end{equation}
where $\bm{\lambda} \triangleq col(\lambda_1,..., \lambda_N)$, $\bm{u}(\bm{x}) \triangleq col(g_1(x_1),...,g_N(x_N))$, $\bm{v}(\bm{x},\bm{\lambda}) \triangleq col(\nabla g_1(x_1)\lambda_1, ..., \nabla g_N(x_N)\lambda_N)$, $\bm{L} \triangleq L \otimes I_{p}$, and $L$ is the Laplacian matrix.

\begin{remark}
{\em
The sub-optimal algorithm given in \cite{Liang2018Singular} for coupled equality constraints is
\begin{equation*}
\left\{%
\begin{aligned}
\dot{\bm{x}} & =  - \nabla f(\bm{x}) - \bm{v}(\bm{x},\bm{\lambda})\\
\varepsilon \dot{\bm{\lambda}} & = \varepsilon \bm{u}(\bm{x}) - \bm{L} \bm{\lambda}
\end{aligned}\right.
\end{equation*}
Our dynamics \eqref{eq:Algorithmcompact2} uses projections to deal with local set constraints and coupled inequalities constraints. Since the projections are not differentiable, some technical difficulties occur in singular perturbation analysis.
}
\end{remark}

\begin{remark}
{\em
Alternative update flows over undirected graphs, referring to \cite{Yi2016Initialization}, can be used as follows.
\begin{equation}\label{eq:algorithmCompair}
\left\{%
\begin{aligned}
\dot{\bm{x}} & =  P_{\bm{\Omega}}(\bm{x} - \nabla f(\bm{x}) - \bm{v}(\bm{x},\bm{\lambda})) - \bm{x}\\
\dot{\bm{\lambda}} & = P_{\mathbb{R}^{pN}_+}(\bm{\lambda} + \bm{u}(\bm{x}) - \bm{L} \bm{\lambda} - \bm{L}\bm{v}) - \bm{\lambda}\\
\dot{\bm{v}} & = \bm{L} \bm{\lambda}
\end{aligned}\right.
\end{equation}
Compared with \eqref{eq:algorithmCompair}, dynamics \eqref{eq:Algorithmcompact2} does not employ the auxiliary variable $\bm{v}$ so that the computation and communication are simplified.
}
\end{remark}
\section{Algorithm analysis}
\label{sec:mainresults}
In this section, we analyze the existence of an equilibrium, the sub-optimality, and the convergence.

\subsection{Existence}\label{subsec:1}
An equilibrium $(\bm{x},\bm{\lambda})$ of Algorithm \ref{alg:2} is a solution to
\begin{subequations}\label{eq:equiAlg}
\begin{align}
\label{eq:equiX}
\bm{0} & =  P_{\bm{\Omega}}(\bm{x} - \nabla f(\bm{x}) - \bm{v}(\bm{x},\bm{\lambda})) - \bm{x}\\
\label{eq:equiLambda}
\bm{0} & = P_{\mathbb{R}^{pN}_+}(\varepsilon \bm{\lambda} + \varepsilon \bm{u}(\bm{x}) - \bm{L} \bm{\lambda}) - \varepsilon \bm{\lambda}
\end{align}
\end{subequations}
which involves projections and nonlinear maps. To show the existence, we first consider the following auxiliary equations
\begin{subequations}\label{eq:equiAlg2}
\begin{align}
\label{eq:equiX2}
\bm{0} & =  P_{\bm{\mathcal{X}}}(\bm{x} - \nabla f(\bm{x}) - \bm{v}(\bm{x},\bm{\lambda})) - \bm{x}\\
\label{eq:equiLambda2}
\bm{0} & = P_{\mathbb{R}^{pN}_+}(\varepsilon \bm{\lambda} + \varepsilon \bm{u}(\bm{x}) - \bm{L} \bm{\lambda}) - \varepsilon \bm{\lambda}
\end{align}
\end{subequations}
where
\begin{equation}\label{eq:constraintX}
\bm{\mathcal{X}} \triangleq \{\bm{x} \in \bm{\Omega} \,|\, g(\bm{x}) \leq \bm{0}\}.
\end{equation}
By \eqref{eq:SOL2Pro}, $\bm{x}$ satisfies \eqref{eq:equiX2} if and only if it is a solution to $\VI(\bm{\mathcal{X}}, \nabla f(\cdot) + \bm{v}(\cdot,\bm{\lambda}))$, regarding $\bm{\lambda}$ as an external input. Also, $\bm{\lambda}$ is a solution to \eqref{eq:equiLambda2} if and only if it is a solution to the generalized equation
\begin{equation*}
\bm{0} \in \varepsilon \bm{u}(\bm{x}) - \bm{L} \bm{\lambda} + \mathcal{N}_{\mathbb{R}_+^{pN}}(\varepsilon\bm{\lambda})= \varepsilon \bm{u}(\bm{x}) - \bm{L} \bm{\lambda} + \mathcal{N}_{\mathbb{R}_+^{pN}}(\bm{\lambda}),
\end{equation*}
which is also equivalent to $\LCP(-\varepsilon \bm{u}(\bm{x}),\bm{L})$, regarding $\bm{x}$ as an external input.
In this way, we can interpret \eqref{eq:equiAlg2} as two interacted static subsystems: One is $\VI$, whose input is $\bm{\lambda}$ and output is
\begin{equation*}
\bm{x} \in \SOL(\bm{\mathcal{X}}, \nabla f(\cdot) + \bm{v}(\cdot,\bm{\lambda})) \triangleq G_1(\bm{\lambda}).
\end{equation*}
The other one is $\LCP$, whose input is $\bm{x}$ and output is
\begin{equation*}
\bm{\lambda} \in \SOL(-\varepsilon \bm{u}(\bm{x}),\bm{L}) \triangleq G_2(\bm{x}).
\end{equation*}
The structure between $G_1$ and $G_2$ is shown in Fig. \ref{fig:structure}.
\begin{figure}[H]
\begin{center}
\includegraphics[width=5cm]{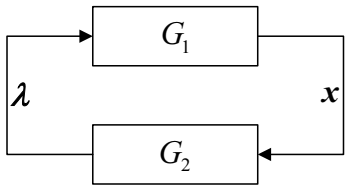}    
\caption{Structure between $G_1$ and $G_2$.} \label{fig:structure}
\end{center}
\end{figure}
Consequently, $(\bm{x}, \bm{\lambda})$ is a solution to \eqref{eq:equiAlg2} if $\bm{x} \in G_1(\bm{\lambda})$ and $\bm{\lambda}\in G_2(\bm{x})$, which leads to fixed-point equations
\begin{equation}\label{eq:fixedPoint}
\bm{x} \in G_1(G_2(\bm{x}))\quad \text{and} \quad \bm{\lambda} \in G_2(G_1(\bm{\lambda})).
\end{equation}
Note that $G_1$ and $G_2$ depend on data of the optimization problem, and $G_2$ also depends on the parameter $\varepsilon$.

\begin{lemma}\label{lem:1}
Under Assumption \ref{assum:1}, $G_1(\bm{\lambda})$ is nonempty and contains only one element for any $\bm{\lambda} \geq \bm{0}$. Moreover,
\begin{equation*}
\|G_1(\bm{\lambda}') - G_1(\bm{\lambda})\| \leq \frac{\kappa_g}{\mu_f}\|\bm{\lambda}' - \bm{\lambda}\|,\quad \forall\,\bm{\lambda}',\bm{\lambda} \geq \bm{0}.
\end{equation*}
\end{lemma}
\begin{proof}
The map $\bm{v}(\cdot,\bm{\lambda})$ with $\bm{\lambda} \geq \bm{0}$ is monotone, since
\begin{multline*}
(\bm{x}' - \bm{x})^T(\bm{v}(\bm{x}',\bm{\lambda}) - \bm{v}(\bm{x},\bm{\lambda})) \\
= \sum_{i=1}^N(\bm{x}_i' - \bm{x}_i)^T(\nabla g_i(\bm{x}') - \nabla g_i(\bm{x}))\bm{\lambda}_i \geq 0.
\end{multline*}
Thus, $\nabla f(\cdot) + \bm{v}(\cdot,\bm{\lambda})$ is $\mu_f$-strongly monotone. As a result, there exists a unique solution to $\VI(\bm{\mathcal{X}},\nabla f(\cdot) + \bm{v}(\cdot,\bm{\lambda}))$, i.e., $G_1(\bm{\lambda})$ is a single-valued map.

Let $\bm{x}' = G_1(\bm{\lambda}')$ and $\bm{x} = G_1(\bm{\lambda})$ for any $\bm{\lambda}',\bm{\lambda} \geq \bm{0}$. By the definition of variational inequality,
\begin{align*}
(\bm{x}' - \bm{x})^T(\nabla f(\bm{x}) + \bm{v}(\bm{x},\bm{\lambda})) \geq 0,\\
(\bm{x} - \bm{x}')^T(\nabla f(\bm{x}') + \bm{v}(\bm{x}',\bm{\lambda}')) \geq 0.
\end{align*}
Therefore,
\begin{equation*}
(\bm{x}' - \bm{x})^T(\nabla f(\bm{x}') + \bm{v}(\bm{x}',\bm{\lambda}') - \nabla f(\bm{x}) - \bm{v}(\bm{x},\bm{\lambda})) \leq 0.
\end{equation*}
By the strongly convexity of $\nabla f(\cdot) + \bm{v}(\cdot,\bm{\lambda})$,
\begin{align*}
\mu_f\|\bm{x}' - \bm{x}\|^2 & \leq (\bm{x}' - \bm{x})^T(\bm{v}(\bm{x}',\bm{\lambda}) - \bm{v}(\bm{x}',\bm{\lambda}'))\\
& \leq \kappa_g \|\bm{x}' - \bm{x}\| \cdot \|\bm{\lambda}' - \bm{\lambda}\|.
\end{align*}
This completes the proof.
\end{proof}

\begin{lemma}\label{lem:2}
Under Assumption \ref{assum:1}, the following statements hold:
\begin{enumerate}[1)]
\item $G_2(\bm{x})$ is nonempty for any $\bm{x} \in \bm{\mathcal{X}}$.
\item $G_2$ has a unique single-valued continuous selection $G_2^\sharp$. That is, $G_2^\sharp$ is a continuous map and $G_2^\sharp(\bm{x}) \in G_2(\bm{x})$ for any $\bm{x} \in \bm{\mathcal{X}}$.
\item There is a constant $\kappa_L > 0$ such that
\begin{equation*}
\|G_2^\sharp(\bm{x}') - G_2^\sharp(\bm{x})\| \leq \varepsilon \kappa_g \kappa_L\|\bm{x}' - \bm{x}\|, \quad \forall\,\bm{x}',\bm{x} \in \bm{\mathcal{X}}.
\end{equation*}
\end{enumerate}
\end{lemma}
\begin{proof}
Consider $\LCP(-u,L)$, where $L$ is the Laplacian matrix. A point $z\in \SOL(-u,L)$ if and only if
\begin{subequations}
\begin{align}
\label{eq:LCPcondi1}
z & \geq \bm{0}\\
\label{eq:LCPcondi2}
Lz - u & \geq \bm{0}\\
\label{eq:LCPcondi3}
z^T(Lz - u) & = 0
\end{align}
\end{subequations}
$\LCP(-u,L)$ is said to be feasible if there exists a point $z\in\mathbb{R}^N$ satisfying \eqref{eq:LCPcondi1} and \eqref{eq:LCPcondi2}, not necessarily satisfying \eqref{eq:LCPcondi3}. It follows from \cite[Theorem 3.1.2]{Cottle2009Linear} that $\SOL(-u,L)$ is nonempty if and only if $\LCP(-u,L)$ is feasible.

Since $L$ has rank $N-1$ and $\bm{1}^T L = \bm{0}^T$, $\LCP(-u,L)$ is feasible if and only if $u\in U \triangleq \{u\in\mathbb{R}^N\,|\,\bm{1}^T u \leq 0\}$. Therefore, $S(u) \triangleq \SOL(-u,L)$ is nonempty for $u\in U$, which implies statement 1).

Let $z',z \in S(u)$. Then
\begin{equation*}
(z'-z)^T(Lz - u) \geq 0\quad \text{and}\quad (z-z')^T(Lz' - u) \geq 0,
\end{equation*}
which implies
\begin{equation*}
(z-z')^TL(z - z') = \frac{1}{2}(z-z')^T(L+L^T)(z - z') \leq 0.
\end{equation*}
Since $L+L^T$ is positive semidefinite, $(L+L^T)(z - z') = 0$, which implies $z' = z + \bm{1}s$ for some $s\in \mathbb{R}$. Also, it follows from \eqref{eq:LCPcondi3} that $s\bm{1}^Tu = 0$. Thus, $S(u)$ is a singleton for $u \in U^\circ \triangleq \{ u \in U\,|\,\bm{1}^Tu < 0\}$, and there is a unique selection map $S^\sharp(u) \in S(u)$ for $u\in U^\circ$.
By \cite[Theorem 7.2.1]{Cottle2009Linear}, there exists a constant $\kappa_L > 0$ depending on $L$ such that for any $u',u \in U$,
\begin{equation*}
S(u') \subseteq S(u) + \kappa_L\|u' - u\|\mathbb{B}.
\end{equation*}
Therefore, $S^\sharp$ is $\kappa_L$-Lipschitz continuous over $U^\circ$ and can be extended to $U$ by taking the limit
\begin{equation*}
S^\sharp(\bar{u}) = \lim_{u\to \bar{u},u\in U^\circ} S^\sharp(u),\quad \forall\, \bar{u}\in U.
\end{equation*}
Thus, statements 2) and 3) hold.
\end{proof}

With Lemmas \ref{lem:1} and \ref{lem:2}, we present the following theorem.

\begin{theorem}\label{thm:1}
Under Assumption \ref{assum:1}, for any $\varepsilon \in (0,\varepsilon^*)$ with $\varepsilon^* \triangleq \frac{\mu_f}{\kappa_g^2\kappa_L}$, there exists an equilibrium $(\bm{x}^*_\varepsilon, \bm{\lambda}^*_\varepsilon)$.
\end{theorem}

\begin{proof}
Since $\varepsilon < \varepsilon^*$, there holds a small gain condition
\begin{equation*}
\frac{\kappa_g}{\mu_f} \cdot \varepsilon \kappa_g\kappa_L <1.
\end{equation*}
Then $G_1(G_2^\sharp(\cdot))$ is a contraction map from $\bm{\mathcal{X}}$ to $\bm{\mathcal{X}}$ and $G_2^\sharp(G_1(\cdot))$ is a contraction map from $\mathbb{R}^{pN}_+$ to $\mathbb{R}^{pN}_+$. Thus, there exists $(\bm{x}_\varepsilon^\dag,\bm{\lambda}_\varepsilon^\dag) \in \bm{\mathcal{X}} \times \mathbb{R}^{pN}_+$ as a solution to \eqref{eq:equiAlg2}.

Next, we construct a solution to \eqref{eq:equiAlg}. Define
\begin{equation*}
f^\dag(\bm{x},\bm{\lambda}) \triangleq \sum_{i=1}^N f_i(x_i) + \lambda_i^Tg_i(x_i).
\end{equation*}
Then $\nabla_{\bm{x}} f^\dag(\bm{x},\bm{\lambda}) = \nabla f(\bm{x}) + \bm{v}(\bm{x},\bm{\lambda})$. By \eqref{eq:equiAlg2}, $\bm{x}^\dag_\varepsilon$ is the optimal solution to
\begin{equation*}
\min_{\bm{x}\in \bm{\Omega}} f^\dag(\bm{x},\bm{\lambda}_\varepsilon^\dag), \,\, \text{ s.t. } \,\, g(\bm{x})\leq \bm{0}.
\end{equation*}
Since the Slater's constraint qualification holds, it follows from Karush-Kuhn-Tucker conditions that there exists a multiplier $\lambda^\ddag \in \mathbb{R}^p$ with $\bm{\lambda}^\ddag \triangleq col(\lambda^\ddag,...,\lambda^\ddag)$ such that
\begin{subequations}
\begin{align}
\label{eq:equiX3}
\bm{0} & \in \nabla f(\bm{x}^\dag_\varepsilon) + \bm{v}(\bm{x}^\dag_\varepsilon,\bm{\lambda}^\dag_\varepsilon + \bm{\lambda}^\ddag) + \mathcal{N}_{\bm{\Omega}}(\bm{x}^\dag_\varepsilon)\\
\label{eq:equiLambda3}
\bm{0} & \leq \lambda^{\ddag} \perp - g(\bm{x}^\dag_\varepsilon) \geq \bm{0}
\end{align}
\end{subequations}
Let $(\bm{x}_\varepsilon^*,\bm{\lambda}_\varepsilon^*) \triangleq (\bm{x}^\dag_\varepsilon, \bm{\lambda}^\dag_\varepsilon + \bm{\lambda}^\ddag)$. It follows from \eqref{eq:equiX3} that $(\bm{x}_\varepsilon^*,\bm{\lambda}_\varepsilon^*)$ renders \eqref{eq:equiX}. Also, it follows from \eqref{eq:equiLambda2} and \eqref{eq:equiLambda3} that $(\bm{x}_\varepsilon^*,\bm{\lambda}_\varepsilon^*)$ renders \eqref{eq:equiLambda}. In other words, $(\bm{x}_\varepsilon^*,\bm{\lambda}_\varepsilon^*)$ is an equilibrium satisfying \eqref{eq:equiAlg}. This completes the proof.
\end{proof}

\begin{remark}
{\em
We first give a solution to \eqref{eq:equiAlg2} and then a solution to \eqref{eq:equiAlg}, by repeatedly taking advantage of variational inequalities. The method is totally different from that given in \cite{Liang2018Singular}. In addition, $\bm{x}^*_\varepsilon \in \bm{\mathcal{X}}$ indicates that the solution satisfies the local and coupled constraints in problem \eqref{eq:optimizationProblem}.
}
\end{remark}
\subsection{Sub-optimality}
The sub-optimality of Algorithm \ref{alg:2} is as follows.
\begin{theorem}\label{thm:2}
Let $\tilde{\varepsilon}^* \triangleq \frac{1}{2} \varepsilon^* = \frac{\mu_f}{2\kappa_g^2 \kappa_L}$ and $K \triangleq \frac{2\kappa_g\kappa_L\|\bm{u}(\bm{x}^*)\|}{\mu_f}$. Then for any $\varepsilon \in (0,\tilde{\varepsilon}^*)$, there holds
\begin{equation}\label{eq:suboptimal}
\|\bm{x}^*_\varepsilon - \bm{x}^*\|\leq K\varepsilon.
\end{equation}
where $\bm{x}^*_\varepsilon$ is given in Theorem \ref{thm:1} and $\bm{x}^*$ is the optimal solution to problem \eqref{eq:optimizationProblem}.
\end{theorem}
\begin{proof}
Since $\bm{x}^*$ is the optimal solution to \eqref{eq:optimizationProblem}, it is also the solution to the variational inequality
$\VI(\bm{\mathcal{X}}, \nabla f(\cdot))$. That is, $\bm{x}^* \in G_1(\bm{0})$. Hence,
\begin{equation*}
\|\bm{x}_\varepsilon^* - \bm{x}^*\| = \|G_1(\bm{\lambda}_\varepsilon^*) - G_1(\bm{0})\| \leq \frac{\kappa_g}{\mu_f}\|\bm{\lambda}_\varepsilon^* - \bm{0}\|.
\end{equation*}
Since $\bm{\lambda}_\varepsilon^* \in \SOL(-\varepsilon \bm{u}(\bm{x}_\varepsilon^*), \bm{L})$ and $\bm{0} \in \SOL(\bm{0},\bm{L})$,
\begin{equation*}
\|\bm{\lambda}_\varepsilon^* - \bm{0}\| \leq \kappa_L\|\varepsilon\bm{u}(\bm{x}_\varepsilon^*) - \bm{0}\|.
\end{equation*}
By the $\kappa_g$-Lipschitz continuity of $\bm{u}(\cdot)$,
\begin{equation*}
\|\bm{u}(\bm{x}_\varepsilon^*)\| \leq \|\bm{u}(\bm{x}^*)\| + \kappa_g \|\bm{x}^*_\varepsilon - \bm{x}^*\|.
\end{equation*}
Therefore,
\begin{equation*}
\bigg(1-\frac{\kappa_g^2}{\mu_f}\kappa_L\varepsilon\bigg)\|\bm{x}_\varepsilon^* - \bm{x}^*\| \leq \frac{\kappa_g}{\mu_f}\kappa_L\varepsilon\|\bm{u}(\bm{x}^*)\|,
\end{equation*}
which implies \eqref{eq:suboptimal}. This completes the proof.
\end{proof}

\begin{remark}
{\em
The expression of $K$ indicates two aspects. First, it shows that the error bound is proportional to $\varepsilon$, since $K$ does not depend on $\varepsilon$. Even the value of $K$ is unknown, one can evaluate that to what extent the accuracy is improved when $\varepsilon$ is reduced. Second, when the Laplacian matrix $L$ is known and the local constrains are bounded, $\kappa_g,\kappa_L$ and the upper bound of $\|\bm{u}(\cdot)\|$ can be estimated offline. In this case, the constant $K$ is available and one can determine the $\varepsilon$ to meet any accuracy of practical use by simple calculation.
}
\end{remark}
\subsection{Convergence}
The update flows \eqref{eq:Algorithmcompact2} can be written as
\begin{equation}\label{eq:algorithmz}
\dot{\bm{z}} = P_{\bm{\Lambda}}(\bm{z} - \bm{G}(\bm{z})) - \bm{z},
\end{equation}
where $\bm{z} \triangleq col(\bm{x},\bm{\lambda}), \bm{\Lambda} \triangleq \bm{\Omega} \times \mathbb{R}_+^{pN}$ and
\begin{equation*}
\bm{G}(\bm{z}) \triangleq \begin{bmatrix}
\nabla f(\bm{x}) + \bm{v}(\bm{\lambda},\bm{v})\\
\frac{1}{\varepsilon} \bm{L}\bm{\lambda} - \bm{u}(\bm{x})
\end{bmatrix}.
\end{equation*}
The map $\bm{G}$ is monotone because
\begin{equation*}
\begin{aligned}
& (\bm{z}' - \bm{z})^T (\bm{G}(\bm{z}') -\bm{G}(\bm{z})) \\
= {}&{} (\bm{x}' - \bm{x})^T (\nabla f(\bm{x}') - \nabla f(\bm{x}) + \bm{v}(\bm{x}',\bm{\lambda}') - \bm{v}(\bm{x},\bm{\lambda}))\\
& + (\bm{\lambda}' - \bm{\lambda})^T(- \bm{u}(\bm{x}') + \bm{u}(\bm{x})) + \frac{1}{\varepsilon}(\bm{\lambda}' - \bm{\lambda})^T\bm{L}(\bm{\lambda}' - \bm{\lambda})\\
\geq {}&{} \mu_f \|\bm{x}' - \bm{x}\|^2 + \frac{1}{\varepsilon}(\bm{\lambda}' - \bm{\lambda})^T\bm{L}(\bm{\lambda}' - \bm{\lambda}), \quad \forall\, \bm{z}',\bm{z}\in \bm{\Lambda}.
\end{aligned}
\end{equation*}
In order to obtain the convergence, we employ a Lyapunov candidate function
\begin{equation*}
V(\bm{z}) \triangleq (\bm{z}-\bm{H}(\bm{z}))^T\bm{G}(\bm{z}) - \frac{1}{2}\|\bm{z}-\bm{H}(\bm{z})\|^2+\frac{1}{2}\|\bm{z}-\bm{z}^*\|^2\text{,}
\end{equation*}
where $\bm{H}(\bm{z}) \triangleq P_{\bm{\Lambda}}(\bm{z} - \bm{G}(\bm{z}))$, and
\begin{equation}\label{eq:Lambdastar}
\bm{z}^* \in \bm{\Lambda}^* \triangleq \{\bm{x}_\varepsilon^*\} \times \{\bm{\lambda}\,|\,(\bm{x}_\varepsilon^*,\bm{\lambda}) \text{ satisfies \eqref{eq:equiAlg}}\}.
\end{equation}

\begin{lemma}\label{lem:3}
Under Assumption \ref{assum:1}, $V(\bm{z})$ is locally Lipschitz continuous in $\bm{\Lambda}$ and is positive definite with respect to $\bm{z}^*$, i.e.,
\begin{equation*}
V(\bm{z}) \geq 0, \, \forall\, \bm{z} \in \bm{\Lambda} \quad \text{and} \quad   V(\bm{z}) = 0 \Leftrightarrow \bm{z}=\bm{z}^*.
\end{equation*}
\end{lemma}
\begin{proof}
By \eqref{eq:pro_Lip}, $\bm{H}$ is locally Lipschitz continuous, which indicates that $V$ is also locally Lipschitz continuous. By calculations, $(\bm{z}-\bm{H}(\bm{z}))^T\bm{G}(\bm{z}) - \frac{1}{2}\|\bm{z}-\bm{H}(\bm{z})\|^2 = - \frac{1}{2}\|\bm{z} - \bm{G}(\bm{z}) - \bm{H}(\bm{z})\|^2 + \frac{1}{2}\|\bm{G}(\bm{z})\|^2 = \max_{\bm{y}\in \bm{\Lambda}} \{- \frac{1}{2}\|\bm{z} - \bm{G}(\bm{z}) - \bm{y}\|^2\} + \frac{1}{2}\|\bm{G}(\bm{z})\|^2 \geq 0$, where the inequality is obtained by letting $\bm{y} = \bm{z}$. Therefore,
\begin{equation*}
V(\bm{z}) \geq \frac{1}{2}\|\bm{z}-\bm{z}^*\|^2, \quad \forall\,\bm{z} \in \bm{\Lambda}.
\end{equation*}
This completes the proof.
\end{proof}

\begin{lemma}\label{lem:4}
Under Assumption \ref{assum:1}, dynamics \eqref{eq:algorithmz} has a unique trajectory $\bm{z}(t) \in \bm{\Lambda}, t\geq 0$. Moreover, the set of equilibria $\bm{\Lambda}^*$ given in \eqref{eq:Lambdastar} is Lyapunov stable.
\end{lemma}
\begin{proof}
Since the right-hand side of \eqref{eq:algorithmz} is locally Lipschitz continuous, there exists a unique trajectory $\bm{z}(t)$. Also, since $\dot{\bm{z}} \in \mathcal{T}_{\bm{\Lambda}}(\bm{z})$, $\bm{z}(t) \in \bm{\Lambda}$ for all $t\geq 0$.

For the Lyapunov stability of \eqref{eq:algorithmz}, it suffices to prove that $V(\bm{z}(t))$ is non-increasing with respect to $t$.
Since $V(\bm{z})$ is locally Lipshcitz continuous and $\bm{z}(t)$ is continuously differentiable, $V(\bm{z}(t))$ is differentiable for almost all $t>0$ with
\begin{multline*}
\dot{V}(\bm{z}(t)) = \dot{\bm{z}}^T (\bm{G}(\bm{z}) + \bm{H}(\bm{z}) - \bm{z}^*) \\
- \lim_{\tau\to 0^+} \frac{\dot{\bm{z}}^T(\bm{G}(\bm{z} + \tau\dot{\bm{z}}) - \bm{G}(\bm{z}))}{\tau}.
\end{multline*}
Since $\bm{G}$ is monotone, $\dot{\bm{z}}^T(\bm{G}(\bm{z} + \tau\dot{\bm{z}}) - \bm{G}(\bm{z})) \geq 0$.
Also,
\begin{equation*}
\dot{\bm{z}}^T (\bm{G}(\bm{z}) + \bm{H}(\bm{z}) - \bm{z}^*) = -(W_1(\bm{z}) + W_2(\bm{z}) + W_3(\bm{z})),
\end{equation*}
where
\begin{align*}
W_1(\bm{z}) & = (\bm{z}^* - \bm{H}(\bm{z}))^T(\bm{H}(\bm{z}) + \bm{G}(\bm{z}) - \bm{z}),\\
W_2(\bm{z}) & = (\bm{z} - \bm{z}^*)^T\bm{G}(\bm{z}^*),\\
W_3(\bm{z}) & = (\bm{z} - \bm{z}^*)^T(\bm{G}(\bm{z}) - \bm{G}(\bm{z}^*)).
\end{align*}
It follows from \eqref{eq:pro_projection} that $W_1(\bm{z}) \geq 0$. Moreover, $W_2(\bm{z}) \geq 0$ because
$\bm{z}^*$ is a solution to the variational inequality $\VI(\bm{\Lambda}, \bm{G})$. Furthermore, $W_3(\bm{z}) \geq 0$ due to the monotonicity of $\bm{G}$. As a result, $\dot{V}(\bm{z}(t))\leq 0$ for almost all $t>0$. This completes the proof.
\end{proof}

The convergence analysis is given in the following result.
\begin{theorem}\label{thm:3}
Under Assumption \ref{assum:1}, for any $\varepsilon \in (0,\varepsilon^*)$, the trajectory of Algorithm \ref{alg:2} converges to an equilibrium point, i.e.,
\begin{equation}\label{eq:convergence}
\lim_{t\to \infty} \bm{z}(t) = \tilde{\bm{z}}^* \in \bm{\Lambda}^*,
\end{equation}
where $\bm{\Lambda}^*$ is given in \eqref{eq:Lambdastar}.
\end{theorem}

\begin{proof}
Since $V(\bm{z}(t))$ is continuous and non-increasing, it follows from the invariance principle that $\bm{z}(t)$ converges to the largest invariant set $\bm{\mathcal{Z}}_{\text{inv}} \subset \bm{\mathcal{Z}}$, where
\begin{multline*}
\bm{\mathcal{Z}} \triangleq \{\bm{z}\in \bm{\Lambda}\,|\, (\bm{z} - \bm{z}^*)^T(\bm{G}(\bm{z}) - \bm{G}(\bm{z}^*)) =0, \\ \text{and } (\bm{z} - \bm{z}^*)^T\bm{G}(\bm{z}^*) =0\}.
\end{multline*}
By the monotonicity of $\bm{G}$, $\bm{z} \in \bm{\mathcal{Z}}$ implies $\bm{x} = \bm{x}^*_\varepsilon$ and $\bm{L}(\bm{\lambda}- \bm{\lambda}^*_\varepsilon) = \bm{0}$. On the one hand, for $\{\bm{x}^*_\varepsilon\}$ being invariant, it is necessary that $\dot{\bm{x}} = 0$ for any $(\bm{x},\bm{\lambda}) \in \bm{\mathcal{Z}}_{\text{inv}}$. Thus, $\bm{\lambda}$ satisfies \eqref{eq:equiX}.
On the other hand, it follows from $(\bm{z} - \bm{z}^*)^T\bm{G}(\bm{z}^*) = 0$ that $(\bm{\lambda}- \bm{\lambda}^*_\varepsilon)^T(\bm{L}\bm{\lambda}^*_\varepsilon- \varepsilon\bm{u}(\bm{x}^*_\varepsilon)) = 0$. Therefore,
\begin{equation*}
\bm{0}\leq \bm{\lambda} \perp \bm{L}\bm{\lambda} - \varepsilon\bm{u}(\bm{x}^*_\varepsilon) \geq \bm{0},
\end{equation*}
which implies that $\bm{\lambda}$ satisfies \eqref{eq:equiLambda2}.
Thus, $\bm{\mathcal{Z}}_{\text{inv}} \subset \bm{\Lambda}^*$.

Let $\tilde{\bm{z}}^*$ be a cluster point of $\bm{z}(t)$ as $t\to +\infty$, i.e., $\tilde{\bm{z}}^*$ is a positive limit point of $\bm{z}(t)$. Then $\tilde{\bm{z}}^* \in \bm{\mathcal{Z}}_{\text{inv}}$ because the positive limit set is invariant \cite[Lemma 4.1]{Khalil2002Nonlinear}.
Redefine a Lyapunov function as
\begin{equation*}
\tilde{V}(\bm{z}) \triangleq (\bm{z}-\bm{H}(\bm{z}))^T\bm{G}(\bm{z}) - \frac{1}{2}\|\bm{z}-\bm{H}(\bm{z})\|^2+\frac{1}{2}\|\bm{z}-\tilde{\bm{z}}^*\|^2.
\end{equation*}
Since $\tilde{\bm{z}}^* \in \bm{\mathcal{Z}}_{\text{inv}} \subset \bm{\Lambda}^*$, it follows from similar arguments in Lemmas \ref{lem:3} and \ref{lem:4} that $\tilde{V}$ is non-increasing along the trajectory $\bm{z}(t)$, and meanwhile, $\tilde{V}(\bm{z}(t))\to 0$ as $t\to +\infty$. Thus, the conclusion follows.
\end{proof}

\begin{remark}
{\em
The convergence analysis is based on Lyapunov functions $V$ and $\tilde{V}$. Similar functions have also been considered in \cite{Yi2016Initialization}, where a derivative formula for $\nabla V$ is needed with the help of $\nabla^2 f$. Here, the convergence analysis does not require the twice differentiability of the cost function.
}
\end{remark}
\section{Numerical experiments}
\label{sec:simulation}
\begin{table*}
\caption{Performance in communication burden, termination time, and relative error of our sub-optimal algorithm and algorithm \eqref{eq:algorithmCompair} over different types of graphs with various network sizes. Here, $t_{\text{ter}} = \infty$ means that the algorithm is divergent.}\label{tab:1}
\vspace{4pt}
\begin{tabular}{cccccccccccc}
  \hline \hline
  \multirow{2}*{$N$} & graph & \multirow{2}*{$d_{\text{mean}}$} & \multirow{2}*{$d_{\text{max}}$} &\multicolumn{2}{c}{algorithm \eqref{eq:algorithmCompair}}
  & \multicolumn{2}{c}{$\varepsilon = 0.1$} & \multicolumn{2}{c}{$\varepsilon = 0.01$} & \multicolumn{2}{c}{$\varepsilon = 0.001$}\\
  \cline{5-6}\cline{7-8}\cline{9-10}\cline{11-12}
  & type & & & $t_{\text{ter}}$ & $e_{\text{rel}}$ & $t_{\text{ter}}$ & $e_{\text{rel}}$ & $t_{\text{ter}}$ & $e_{\text{rel}}$ & $t_{\text{ter}}$ & $e_{\text{rel}}$ \\ \hline
  \multirow{3}*{$10$}
&	circle	&	2	&	2	&	$\infty$	&	-	&	$12.384$	&	$7.4768\%$	&	$12.697$	&	$0.9062\%$	&	$12.906$	&	$0.0929\%$	\\
&	random	&	8	&	11	&	$180.108$	&	$0.0008\%$	&	$12.615$	&	$9.0475\%$	&	$12.686$	&	$1.1907\%$	&	$12.9$	&	$0.1233\%$	\\
&	complete	&	18	&	18	&	$31.836$	&	$0.0003\%$	&	$12.36$	&	$3.5692\%$	&	$12.72$	&	$0.4063\%$	&	$12.909$	&	$0.0419\%$	\\
\hline
\multirow{3}*{$50$}
&	circle	&	2	&	2	&	$69.404$	&	$0.0002\%$	&	$13.51$	&	$1.3965\%$	&	$13.631$	&	$0.1627\%$	&	$13.669$	&	$0.0166\%$	\\
&	random	&	47.96	&	57	&	$124.801$	&	$0.0004\%$	&	$13.508$	&	$2.0427\%$	&	$13.617$	&	$0.2543\%$	&	$13.667$	&	$0.0261\%$	\\
&	complete	&	98	&	98	&	$26.597$	&	$0.0002\%$	&	$13.543$	&	$0.8140\%$	&	$13.651$	&	$0.0883\%$	&	$13.671$	&	$0.009\%$	\\
\hline
\multirow{3}*{$100$}
&	circle	&	2	&	2	&	$90.799$	&	$<0.0001\%$	&	$13.903$	&	$1.9957\%$	&	$14.021$	&	$0.2295\%$	&	$14.062$	&	$0.0233\%$	\\
&	random	&	97.98	&	112	&	$705.96$	&	$0.0006\%$	&	$13.899$	&	$4.7095\%$	&	$13.963$	&	$0.7167\%$	&	$14.052$	&	$0.0759\%$	\\
&	complete	&	198	&	198	&	$27.46$	&	$0.0001\%$	&	$13.923$	&	$1.1618\%$	&	$14.042$	&	$0.1257\%$	&	$14.065$	&	$0.0127\%$	\\
\hline
\multirow{3}*{$500$}
&	circle	&	2	&	2	&	$44.782$	&	$<0.0001\%$	&	$14.875$	&	$0.0077\%$	&	$14.876$	&	$0.0009\%$	&	$14.877$	&	$<0.0001\%$	\\
&	random	&	497.808	&	542	&	$1743.164$	&	$0.0007\%$	&	$15.127$	&	$0.0314\%$	&	$14.875$	&	$0.0078\%$	&	$14.876$	&	$0.0009\%$	\\
&	complete	&	998	&	998	&	$20.725$	&	$<0.0001\%$	&	$14.875$	&	$0.0042\%$	&	$14.88$	&	$0.0005\%$	&	$14.877$	&	$<0.0001\%$	\\
\hline
\multirow{3}*{$1000$}
&	circle	&	2	&	2	&	$\infty$	&	-	&	$22.572$	&	$8.8231\%$	&	$22.185$	&	$2.5975\%$	&	$15.01$	&	$0.6054\%$	\\
&	random	&	998.572	&	1065	&	$>2000$	&	$7.4794\%$	&	$23.487$	&	$19.4877\%$	&	$15.206$	&	$6.2969\%$	&	$14.716$	&	$0.9531\%$	\\
&	complete	&	1998	&	1998	&	$53.428$	&	$<0.0001\%$	&	$14.65$	&	$3.0983\%$	&	$14.987$	&	$0.3729\%$	&	$15.21$	&	$0.0385\%$	\\
  \hline \hline
\end{tabular}
\end{table*}
Consider a virtualized 5G system consisting of $N$ slices \cite{Halabian2019D5G}. Each slice shares $M$ virtual network functions (VNFs), which are being distributed over $K$ data centers (DCs). Each DC provides resources such as CPU, RAM, bandwidth, and storage. The amount of these $\ell$ types of resources are denoted by vectors $R_{k}\in \mathbb{R}^\ell, k \in \{1,2,...,K\}$. Also, each slice $i\in \{1,2,...,N\}$ is associated with a set of demand vectors denoted by $d^{k,m}_i \in \mathbb{R}^\ell$ for each DC $k$ and each VNF $m$. The optimization problem is to determine the amount of resources allocated to each of the VNFs in each DC by minimizing the sum of cost functions of slice thicknesses subjected to resource constraints, i.e.,
\begin{equation*}
\begin{array}{ll}
\mathop{\text{Minimize}}\limits_{\bm{x}} & \sum_{i = 1}^N f_i(x_i), \quad f_i(x_i) = \frac{1}{2}(x_i - \alpha_i)^2\\
\text{Subject to} & \sum_{i=1}^N\sum_{m=1}^Mx_id^{k,m}_i\leq R_{k},\quad k = 1,2,...,K\\
& x_i \geq 0, \quad i = 1,2,...,N
\end{array}
\end{equation*}
Set $\ell = K = M = 1$ and $N = 10, 50, 100, 500, 1000$ with {\em directed circles}, {\em random digraphs}, and {\em complete graphs}, respectively.
Generate randomly $\alpha_i\in [0.5, 2], d_i^{1,1}\in [0,1], R_1\in [0.5N, 2N]$ for $i = 1,2,...,N$.
Set tolerance $\epsilon = 10^{-5}$ with the stopping criterion
\begin{equation*}
\|\dot{\bm{z}}(t)\| \leq \epsilon,
\end{equation*}
where $\dot{\bm{z}}(t)$ was given in \eqref{eq:algorithmz}. Record the termination time, denoted by $t_{\text{ter}}$, and calculate the relative error
\begin{equation*}
e_{\text{rel}} = \frac{\|\bm{x}(t_{\text{ter}}) - \bm{x}^*\|}{\|\bm{x}^*\|} \times 100\%.
\end{equation*}
The instant communication burden of an agent can be characterized by the number of times that it sends and receives information in a unit running time, which equals the sum of its out-degree and in-degree. We record the mean and maximum of such degrees among all agents, denoted by $d_{\text{mean}}$ and $d_{\text{max}}$, respectively.
The total amount of communication per agent can be evaluated by using $d_{\text{mean}}\cdot t_{\text{ter}}$ and $d_{\text{max}}\cdot t_{\text{ter}}$ for our algorithm and $2d_{\text{mean}}\cdot t_{\text{ter}}$ and $2d_{\text{max}}\cdot t_{\text{ter}}$ for algorithm \eqref{eq:algorithmCompair}.
Note that these two algorithms do not necessarily share the same termination time $t_{\text{ter}}$, because their convergence speed may be different.
In the experiments, the Euler's method is employed to discretize these algorithms with fixed stepsize $0.001$, and the Laplacian matrices are normalized by scaling the balanced weights such that $\|L\| = 1$. Numerical results in Table \ref{tab:1} show that our algorithm achieves acceptable accuracy, fast convergence speed, and significant reduction of computation and communication burden.
\section{Conclusions}
\label{sec:conclusions}
A distributed continuous-time algorithm has been proposed for resource allocation optimization with local set constraints and coupled inequality constraints over weight-balanced graphs. Existence and sub-optimality of the equilibrium have been established with convergence analysis. Our algorithm and analysis approach have demonstrated the effectiveness of the singular perturbation based sub-optimal design even with non-differentiable right-hand side.
\bibliographystyle{dcu}        
\bibliography{../../../../bib/refference0}
\end{document}